\documentclass[psamsfonts]{amsart}

\usepackage{amssymb,amsfonts}
\usepackage{url}
\usepackage[all,arc]{xy}
\usepackage{enumerate}
\usepackage{mathrsfs}
\usepackage{tikz-cd}
\newtheorem{thm}{Theorem}[section]  
\newtheorem{cor}[thm]{Corollary}
\newtheorem{prop}[thm]{Proposition}
\newtheorem{lem}[thm]{Lemma}

\theoremstyle{definition}
\newtheorem{defn}[thm]{Definition}

\newtheorem{exmp}[thm]{Example}

\theoremstyle{remark}
\newtheorem{rem}[thm]{Remark}

\makeatletter
\let\c@equation\c@thm
\makeatother
\numberwithin{equation}{section}

\newcommand{\R}{\mathbb{R}}

\newcommand{\N}{\mathbb{N}}
\newcommand{\Z}{\mathbb{Z}}
\newcommand{\C}{\mathbb{C}}
\newcommand{\HH}{\mathbb{H}}

\renewcommand\Im{\operatorname{Im}}

\newcommand{\SO}{\text{SO}}

\newcommand{\SL}{\text{SL}}

\newcommand{\fsl}{\mathfrak{sl}}

\newcommand{\PSL}{\text{PSL}}
\newcommand{\tr}{\text{Tr}}
\newcommand{\conv}{\text{conv}}

\title{Metrics and compactifications of Teichm\"uller spaces of flat tori}

\author{Mark Greenfield and Lizhen Ji}

\date{}

\begin{document}

\begin{abstract}
Using the identification of the symmetric space $\SL(n,\R)/\SO(n)$ with the Teichm\"uller space of flat $n$-tori of unit volume, we explore several metrics and compactifications of these spaces, drawing inspiration both from Teichm\"uller theory and symmetric spaces. We define and study analogs of the Thurston, Teichm\"uller, and Weil-Petersson metrics. We show the Teichm\"uller metric is a symmetrization of the Thurston metric, which is a polyhedral Finsler metric, and the Weil-Petersson metric is the Riemannian metric of $\SL(n,\R)/\SO(n)$ as a symmetric space. We also construct a Thurston-type compactification using measured foliations on $n$-tori, and show that the horofunction compactification with respect to the Thurston metric is isomorphic to it, as well as to a minimal Satake compactification. 
\end{abstract}

\maketitle

\tableofcontents

\section{Introduction}
\label{intro}

Throughout their long histories, there has been a great deal of work studying analogies between Teichm\"uller spaces and symmetric spaces. Usually, questions and results about the latter motivate those about the former. In this paper, we reverse this pattern and use a modular interpretation of the symmetric space $\SL(n,\R)/\SO(n)$ to define and interpret new and old metrics and compactifications. 

While there are similarities between the action of mapping class groups on Teichm\"uller spaces and the action of arithmetic subgroups of Lie groups on associated symmetric spaces, Teichm\"uller spaces are very far from being symmetric spaces. For example, a corollary of Royden's theorem \cite{royden} shows that there are no symmetric points of Teichm\"uller spaces. Despite important departures from symmetric space behavior, in the case of flat $n$-tori of unit volume, the Teichm\"uller spaces are precisely symmetric spaces. 

The Teichm\"uller space of a closed oriented surface $S_g$ of genus $g$, denoted  $\mathcal{T}(S_g)$, is the moduli space of marked complex structures on the surface. By the uniformization theorem each such marked complex structure possesses a canonical Riemannian metric of constant curvature. Several different metrics have been defined for $\mathcal{T}(S_g)$, including the classical Teichm\"uller metric $d_{Teich}$, defined in terms of extremal quasiconformal distortion between two marked complex structures. Another well-known metric on $\mathcal{T}(S_g)$ is the Weil-Petersson metric, introduced by Weil \cite{weil}, which is an incomplete Riemannian metric. 

In \cite{thurston}, Thurston defined an asymmetric metric on $\mathcal{T}(S_g)$, $g\geq2$, using the extremal Lipschitz constant for marking-preserving maps between hyperbolic surfaces. This metric is natural for Teichm\"uller spaces of hyperbolic surfaces as it uses only the canonical Riemannian metric associated to each complex structure.

In this paper, after defining the Teichm\"uller spaces of unit volume flat $n$-tori, denoted by $\mathcal{T}(n)$, we will define analogs of these three metrics for $\mathcal{T}(n)$. The natural bijection $\mathcal{T}(n)\leftrightarrow\SL(n,\R)/\SO(n)$ (reviewed in Section \ref{teichmullertori}) is utilized throughout. 

\begin{thm}
\label{metricsummary}
For $\mathcal{T}(2)$, the Thurston metric, Teichm\"uller metric, Weil-Petersson metric, and hyperbolic metric all coincide. For $\mathcal{T}(n)$ with $n\geq3$, we have:
\begin{enumerate}
\item The Thurston metric is an asymmetric polyhedral Finsler metric which can be computed explicitly (Theorem \ref{lipschitzmetricfmla}, Proposition \ref{thurstonhoroboundary}).
\item The Teichm\"uller metric is the symmetrization of the Thurston metric by maximum (Theorem \ref{symofthurst}).
\item The Weil-Petersson metric is equal to the natural Riemannian metric on the symmetric space $\SL(n,\R)/\SO(n)$ (Proposition \ref{wpmetric}). 
\end{enumerate}
\end{thm}

In addition, the Teichm\"uller metric on $\mathcal{T}(n)$ has been studied in a very different context before: in \cite{hilbertmetric}, the same metric on $\SL(n,\R)/\SO(n)$ was found to be a generalization of the Hilbert projective metric.

Our main tool for understanding the Thurston metric is Proposition \ref{extremalaffine}, where we show that the minimal Lipschitz constant is realized by the unique affine map between two marked tori. Recall that the extremal quasiconformal map realizing the Teichm\"uller distance is unique (see Theorem 11.9 of \cite{primer}, originally in \cite{teich}). Interestingly, this is not the case for extremal Lipschitz maps. We give a construction for an infinite family of extremal Lipschitz maps in Proposition \ref{nonuniqueaffine}.

Compactifications of symmetric spaces are well-studied from many perspectives. One of the most important constructions is the Satake compactification associated to a representation of the isometry group, first studied in \cite{satake}. Another is the horofunction compactification with respect to a (Finsler) metric, first defined by Gromov in \cite{gromov}. 

Compactifications of Teichm\"uller spaces have also been extensively studied. Thurston's compactification and its geometric interpretation using projective measured foliations (see \cite{fathi}) is the most well-known. In \cite{walsh}, Walsh showed that the horofunction compactification with respect to the Thurston metric is equivalent to Thurston's compactification. 

Haettel in \cite{haettel} has defined and studied a Thurston-type compactification of the space of marked lattices in $\R^n$ via an embedding in the projective space $\mathbb{P}(\R^{\Z^n}_+)$. This mimics the original construction of Thurston. Theorem 3.1 in \cite{haettel} shows that this compactification is $\SL(n,\R)$-equivariantly isomorphic to the minimal Satake compactification induced by the standard representation of $\SL(n,\R)$. 

In Section \ref{thurstoncompact}, we introduce a related compactification of $\mathcal{T}(n)$, analogous to the geometric description of Thurston's compactification. In particular, we define an analog of projective measured foliations on $n$-tori to construct a Thurston boundary of $\mathcal{T}(n)$. 

\begin{thm}
\label{horogeomcompact}
For the Teichm\"uller space $\mathcal{T}(n)=\SL(n,\R)/\SO(n)$ of unit volume flat $n$-tori, the following compactifications are $\SL(n,\R)$-equivariantly isomorphic:
\begin{enumerate}
\item Thurston compactification via measured foliations on $n$-tori
\item Horofunction compactification with respect to the Thurston metric
\item Minimal Satake compactification associated to the standard representation of $\SL(n,\R)$
\end{enumerate}
\end{thm}

The equivalence (1)$\leftrightarrow$(2) is analogous to the case of hyperbolic surfaces, while (1)$\leftrightarrow$(3) is related to Theorem 3.1 in \cite{haettel}, and gives a geometric interpretation of the boundary points of the compactification in \cite{haettel}. Theorem \ref{horogeomcompact} is the combination of Proposition \ref{thurstonhoroboundary} and Theorem \ref{thurstonGisom}. We also show the following for the Teichm\"uller metric: 

\begin{thm}
\label{teichhorocompact}
The horofunction compactification of $\mathcal{T}(n)$ with the Teichm\"uller metric is $\SL(n,\R)$-equivariantly isomorphic to the generalized Satake compactification associated to the sum of the standard and dual representations of $\SL(n,\R)$.
\end{thm}

Finally, as an immediate corollary to Theorem \ref{metricsummary}(3) and well-known facts about compactifications of nonpositively curved Riemannian symmetric spaces, we have:
\begin{cor}
The horofunction compactification of $\mathcal{T}(n)$ with the Weil-Petersson metric is the visual compactification.
\end{cor}

This work began by considering the Thurston metric on Teichm\"uller spaces of 2-tori, following \cite{bpt}. By defining a new analog of Thurston's metric and extending to higher dimensions, this work (especially Theorem \ref{lipschitzmetricfmla}) gives an answer to Problem 5.3 in W. Su's list of problems on the Thurston metric \cite{problems} from the AIM workshop ``Lipschitz metric on Teichm\"uller space" in 2012. 
\newline

\noindent{\bf Acknowledgements:} The authors wish to thank Richard Canary for several helpful discussions and Athanase Papadopoulos for suggesting some important references. The first author is supported by the National Science Foundation Graduate Research Fellowship Program under Grant No. DGE\#1256260.

\section{Teichm\"uller spaces of hyperbolic surfaces}
\label{teichbackground}

Here, we will review some background on Teichm\"uller spaces for hyperbolic surfaces. Let $S_g$ be a closed, oriented smooth surface of genus $g\geq2$. 

\begin{defn}
\label{classical}
The Teichm\"uller space $\mathcal{T}(S_g)$ is defined as the set of equivalence classes of marked closed Riemann surfaces of genus $g$:
$$
\mathcal{T}(S_g) = \{[S,f]: S \text{ a Riemann surface}, f:S \to S_g \text{ orientation-perserving homeomorphism}\}/\sim
$$
where $[S,f]\sim[S',f']$ if and only if there exists a biholomorphism $h$ such that the following diagram commutes up to homotopy:

\begin{figure}[!ht]
\centering
\begin{tikzcd}[row sep=tiny]
S  \arrow[dd, "h"]\arrow[rd, "f"] \\
 & S_g   \\
S' \arrow[ur, "f'"]
\end{tikzcd}
\end{figure}
\end{defn}

\begin{rem} 
By forgetting the maps $f$ and $f'$, we forget the markings and the condition reduces to conformal equivalence. The resulting collection defines the moduli space $\mathcal{M}_g$ of complex structures on $S_g$. More formally, the moduli space is realized as the quotient $\mathcal{M}_g=\text{Mod}_g\backslash\mathcal{T}(S_g),$ where $\text{Mod}_g = \text{Diff}^+(S_g)/\text{Diff}_0(S_g)$ is the mapping class group of $S_g$, and the action is given by $$\varphi\cdot[S,f] = [S,\varphi\circ f]$$
\end{rem}

Recall next the correspondence between complex structures and constant-curvature metrics via the uniformization theorem. 

\begin{prop}
\label{canonicalbij}
For each $g\geq2$, there is a canonical bijection
$$
\mathcal{T}(S_g)\cong \text{Met}_g/\text{Diff}_0(S_g)
$$
where $\text{Met}^{-1}_g$ is the collection of hyperbolic metrics on $S_g$, and $\text{Diff}_0(S_g)$ is the collection of diffeomorphisms of $S_g$ isotopic to the identity. 
\end{prop}

This is a special case of Theorem 1.8 in \cite{imayoshi}. We can thus also view Teichm\"uller space as equivalence classes of marked hyperbolic surfaces.

We will next define the Teichm\"uller metric. Let $[S,f],[S',f']\in\mathcal{T}(S_g)$. Then the map $f'^{-1}\circ f$ is an orientation-preserving homeomorphism from $S$ to $S'$. Recall that the quasiconformal dilatation of an orientation-preserving almost-everywhere real differentiable map $\phi:D_1\to D_2$ between domains in $\C$ is given by:
\begin{equation}
\label{qcdef}
K_{\phi} = \sup\frac{|\phi_z|+|\phi_{\bar{z}}|}{|\phi_z|-|\phi_{\bar{z}}|},
\end{equation}
where the supremum is over all points where $\psi$ is real-differentiable. This definition extends to maps between Riemann surfaces. Then the \textit{Teichm\"uller metric} on $\mathcal{T}(S_g)$ is defined as:
\begin{equation}
\label{dteichdef}
d_{Teich}([S,f],[S',f']) = \frac{1}{2}\log \inf_{\phi\in[f'^{-1}\circ f]}(K_{\phi})
\end{equation}
where the infimum is taken over all homeomorphisms $\phi$ in the homotopy class $[f'^{-1}\circ f]$ which are smooth except at finitely many points. One can show this defines a metric on $\mathcal{T}(S_g)$ (see \S5.1 of \cite{imayoshi}). 

For $g=1$, the Teichm\"uller metric was determined by Teichm\"uller in \cite{teich} (see also the translation and commentary in \cite{papadopouloscommentary}):
\begin{prop}
\label{h2t1equiv}
Under the identification $\HH^2\to\mathcal{T}(S_1)$ defined by $\tau\mapsto \C/(\Z+\tau\Z)$, the Teichm\"uller metric is equal to the hyperbolic metric.
\end{prop}

Thurston's (asymmetric) metric \cite{thurston} utilizes the hyperbolic structure on surfaces. If $[S,f],[S',f']\in\mathcal{T}(S_g)$, then the Thurston distance between them is defined:
$$
d_{Th}([S,f],[S',f']) = \frac{1}{2}\inf_{\phi\in[f'^{-1}\circ f]}\log(\mathcal{L}(\phi))
$$
where the infimum is over all Lipschitz maps $\phi:S\to S'$ in $[f'^{-1}\circ f]$, and 
$$
\mathcal{L}(\phi) = \sup_{x\neq y}\frac{d_{S'}(\phi(x),\phi(y))}{d_S(x,y)}
$$
is the Lipschitz constant for $\phi$, and $d_{S'}$, $d_S$ are the induced hyperbolic metrics. 

Lastly, we recall the Weil-Petersson metric \cite{weil}. See also Chapter 7 of \cite{imayoshi} or \cite{hubbard} \S7.7. Let $[S,f]\in\mathcal{T}_g$, and let $Q(S)$ be the vector space of holomorphic quadratic differentials on $S$, identified with the cotangent space of $\mathcal{T}(S_g)$. For $q_1,q_2\in Q(S)$ define an inner product on $Q(S)$ by
$$
\langle q_1,q_2\rangle_{WP} = \int_S\bar{q_1}q_2(ds^2)^{-1},
$$
where $ds^2$ is the hyperbolic metric on the Riemann surface. This induces an inner product on the tangent space $T_{[S,f]}\mathcal{T}(S_g)$, known as the Weil-Petersson metric.

\section{The Teichm\"uller spaces of flat $n$-tori}
\label{teichmullertori}

We will introduce now the Teichm\"uller spaces of unit volume flat $n$-tori, denoted $\mathcal{T}(n)$, where $n\geq2$. Let $\mathbb{T}^n = \R^n/\Z^n$ be the square torus of dimension $n$. 

\begin{defn}
The Teichm\"uller space $\mathcal{T}(n)$ is defined as the set of equivalence classes of marked flat tori of dimension $n$ and unit volume:
$$
\mathcal{T}(n) = \{[S,f]: S\text{ a flat $n$-torus of volume 1, } f:S\to \mathbb{T}^n\text{ orientation-preserving homeo}\}/\sim
$$
where $[S,f]\sim[S',f']$ if and only if there exists an isometry $h:S\to S'$ such that the following diagram commutes up to homotopy:
\begin{figure}[!ht]
\centering
\begin{tikzcd}[row sep=tiny]
S  \arrow[dd, "h"]\arrow[rd, "f"] \\
 & \mathbb{T}^n   \\
S' \arrow[ur, "f'"]
\end{tikzcd}
\end{figure}
\end{defn}

We now recall a few classical facts.

\begin{prop}
There is a natural bijective correspondence between the following spaces:
$$
\mathcal{T}(n)\leftrightarrow\SL(n,\R)/\SO(n).
$$
\end{prop}
\begin{proof}

We use methods similar to \S10.2 of \cite{primer}. Given a marked unit volume torus $f: S\to \mathbb{T}^n$, write $S=\R^n/\Lambda$ for a lattice $\Lambda$ of unit covolume. Lift the map $f$ to $\tilde{f}:\R^n\to\R^n$, and let $\zeta_i = \tilde{f}^{-1}(e_i)$ for $i=1,\ldots,n$, where the $e_i$ are the standard basis vectors of $\R^n$. These form an ordered generating set (i.e. a marking) for the lattice $\Lambda$, the coordinates of which form the columns of a matrix in $\SL(n,\R)$. The original choice of $\Lambda$ was unique up to the action of $\SO(n)$ on $\R^n$, and so this specifies an element of $\SL(n,\R)/\SO(n)$. Homotopic markings give the same lattice by Lemma \ref{lehtohomotopy} below. 

Conversely, given a matrix in $\SL(n,\R)$, the columns form an ordered generating set for a unit covolume lattice $\Lambda$. Now, there exists a linear map $\tilde{\phi}:\R^n\to\R^n$ which sends the ordered generating set for $\Lambda$ to the standard basis of $\R^n$. This map descends to a map $\phi:\R^n/\Lambda \to \mathbb{T}^n$ which defines a marked flat torus. Two matrices will give the same marked flat torus if and only if they represent the same coset in $\SL(n,\R)/\SO(n)$. 
\end{proof}

\begin{cor}
There is a natural bijective correspondence $$\mathcal{T}(2)\leftrightarrow\HH^2.$$ 
\end{cor}
\begin{proof}
We need only the identification $\SL(2,\R)/\SO(2)\leftrightarrow\HH^2$, which follows from the fact that $\SL(2,\R)$ acts transitively on $\HH^2$ by fractional linear transformations with point stabilizers isomorphic to $\SO(2)$.
\end{proof}

\begin{lem}
\label{lehtohomotopy}
\begin{enumerate}
\item The group of isometries of a flat $n$-torus acts transitively.
\item If two homeomorphisms $\varphi_i:S\to S'$, $i=0,1$, between flat $n$-tori are homotopic, then they induce the same isomorphism of deck transformation groups acting on $\R^n$.
\end{enumerate}
\end{lem}

See \cite[Lemma V.6.2, Theorem IV.3.5]{lehto} for the dimension 2 case of Lemma \ref{lehtohomotopy}, whose proofs generalize immediately. Next, we consider the metric perspective on $\mathcal{T}(n)$. 

\begin{prop}\label{slpn}
There is a natural bijective correspondence between the quotient $\SL(n,\R)/\SO(n)$ and the space $\mathcal{P}_n$.
\end{prop}
\begin{proof}
Let $X\in\mathcal{P}_n$. $\SL(n,\R)$ acts on $\mathcal{P}_n$ by $g\cdot X = gXg^T$, where $g^T$ is the transpose. This is transitive with the stabilizer of the identity matrix precisely $\SO(n)$. Hence $\SL(n,\R)/\SO(n)$ is identified with $\mathcal{P}_n$ as homogeneous spaces of $\SL(n,\R)$ by the map $gK\mapsto gg^T$. 
\end{proof}

Henceforth we will interchangeably refer to points of $\mathcal{T}(n)$ as either marked flat $n$-tori, coset (representatives) $gK\in\SL(n,\R)/\SO(n)$, or as elements of $\mathcal{P}_n$. 

While the columns of a matrix representative of a point $gK$ determine a marked lattice $\Lambda$ which descends to a marked flat torus $\R^n/\Lambda$, the corresponding point $gg^T\in\mathcal{P}_n$ also has a concrete interpretation in the language of flat tori. The matrix $gg^T$ is an explicit realization of the metric tensor for $\R^n/\Lambda$. To see this, use Euclidean coordinates on the standard torus $\mathbb{T}^n = \R^n/\Z^n$. The inner product between two vectors $v_1,v_2\in\R^n \cong T_pX$ for any $p\in X = \R^n/\Lambda$ is given by:
$$
\langle v_1,v_2\rangle_p = \langle v_1g,v_2g\rangle_{\R^n} = \langle v_1gg^T,v_2\rangle_{\R^n}.
$$
This defines a Riemannian metric on the standard torus $\R^n/\Z^n$ which is isometric to $\R^n/\Lambda$. If $\gamma:[0,1]\to\R^n/\Z^n$ is a smooth closed curve, then the length $\ell_X(\gamma)$ is computed as follows:
$$
\ell_X(\gamma) = \int_0^1\sqrt{\langle \gamma'(t)X,\gamma'(t)\rangle} dt
$$
This formula behaves nicely with the action $g\cdot \gamma = \gamma g$ for $g\in\SL(n,\R)$:
$$
\ell_X(g\cdot\gamma) = \int_0^1\sqrt{\langle (\gamma'(t)g)X,\gamma'(t)g\rangle} dt = \int_0^1\sqrt{\langle \gamma'(t)(gXg^T),\gamma'(t)\rangle} dt = \ell_{g\cdot X}(\gamma).
$$

\section{Extremal Lipschitz maps between tori}
\label{extremallipschitz}

Let $[S,f],[S',f']\in\mathcal{T}(n)$, with $S=\R^n/\Lambda$ and $S'=\R^n/\Lambda'$. Our main result in this section is the following:

\begin{prop}
\label{extremalaffine}
The map $\psi:S\to S'$ which lifts to the unique affine map $\tilde{\psi}:\R^n\to\R^n$ realizes the minimal Lipschitz constant in $[f'^{-1}\circ f]$.
\end{prop}
\begin{proof}
Let $S=\R^n/\Lambda$ and $S'=\R^n/\Lambda'$ be tori of volume 1 with markings $f$ and $f'$. Because affine self-maps on flat tori are isometric and transitive we may assume lifts of maps $\varphi:S\to S'$ to $\R^n$ have the property that $\tilde{\varphi}(0)=0$. Let $\mathcal{F}$ denote the class of all such lifts whose quotients are homotopic to $f'^{-1}\circ f$. For $g\in \mathcal{F}$, let $\bar{g}$ denote the induced map $S\to S'$. 

Let $q$ and $q'$ be the quotient maps for $S$ and $S'$, respectively. Then for all $g\in \mathcal{F}$, the following diagram commutes:

\begin{figure}[!ht]
\centering
\begin{tikzcd}
\mathbb{R}^n \arrow[r, "g"] \arrow[d, "q"] & \mathbb{R}^n \arrow[d, "q'"] \\
S \arrow[r, "\bar{g}"] & S' 
\end{tikzcd}
\end{figure}

 Let $\{\omega_1,\ldots,\omega_n\}$ be a basis of $\Lambda$. For any $g_1,g_2\in \mathcal{F}$, it follows that $g_1(\omega_i)=g_2(\omega_i)+\lambda_i$ for some $\lambda_i\in\Lambda$ for each of $i=1,\ldots,n$. By Lemma \ref{lehtohomotopy}, it follows that $\lambda_i=0$ for $i=1,\ldots,n$ since $g_1$ and $g_2$ are homotopic. One then obtains a basis $\{\zeta_1,\ldots,\zeta_n\}$ of $\Lambda'$ such that $\mathcal{F}$ is the class of homeomorphisms $g:\R^n\to\R^n$ with
\begin{equation}
\label{gclass}
g(0)=0,\ g(x+\sum_i^n m_i\omega_i) = g(x) + \sum_i^nm_i\zeta_i
\end{equation}
for all $x\in\R^n$. Notice that any homeomorphism $\R^n\to\R^n$ satisfying Equation \ref{gclass} descends to a map $S\to S'$ homotopic to $f'^{-1}\circ f$. The condition of being affine uniquely determines such a map inside a fundamental domain of $\Lambda$, and hence on all of $\R^n$. This proves uniqueness of the affine map; let $w\in \mathcal{F}$ be the affine map.

Now we show $w$ has the least Lipschitz constant. Let $g\in \mathcal{F}$ be a $K$-Lipschitz map, i.e.
\begin{equation}
\label{lipschitzconstant}
K \geq \sup_{x\neq y}\frac{|g(x)-g(y)|}{|x-y|}.
\end{equation}
Define $g_k(x) = g(kx)/k$ for $k=1,2,\ldots$. These maps are all $K$-Lipschitz and satisfy Equation \ref{gclass}, so $g_k\in\mathcal{F}$ for all $k$. By Lemma \ref{uniformconv} below, $g_k\xrightarrow{k\to\infty}w$ uniformly on $\R^n$. It is a standard fact from real analysis that the pointwise limit of a sequence of $K$-Lipschitz functions is also $K$-Lipschitz. Hence $w$ is $K$-Lipschitz. In other words, $K\geq\mathcal{L}(w)$. Because this holds for any Lipschitz map in $\mathcal{F}$, it follows that $w$ has minimal Lipschitz constant. 
\end{proof}

\begin{lem}
\label{uniformconv}
In the proof of Proposition \ref{extremalaffine}, the sequence $g_k\to w$ uniformly.
\end{lem}
\begin{proof}
Pick $\epsilon>0$ and let $x_0\in\R^n$. Since $\omega_1,\ldots,\omega_n$ are linearly independent, $x_0$ may be written as 
$$
x_0=\sum_{i=1}^n r_i\omega_i
$$
for some $r_i\in\R$, $i=1,\ldots,n$. Let 
$$
M=\sup_{(a_1,\ldots,a_n)\in[0,1]^n}\bigg|g(\sum_{i=1}^n a_i\omega_i)\bigg| + \sum_i^n|\zeta_i|.
$$
This is finite since $g$ is continuous and this domain is compact. Then for any integer $k>M/\epsilon$, we have:
\begin{equation}
\label{uniformconvergence}
 \big|g_k(x_0)-w(x_0)\big| = \frac{1}{k}\big|g(k\sum_{i=1}^nr_i\omega_i) - (k\sum_{i=1}^nr_i\zeta_i)\big|
\end{equation}
since $w$ is affine. Write $kr_i=m_i+t_i$, where $t_i\in[0,1)$ and $m_i\in \Z$, for $i=1,\ldots,n$. In Equation \ref{uniformconvergence}, the integer part $m_i$ of each term $kr_i$ factors through $g$. We then compute:
$$
 \big|g_k(x_0)-w(x_0)\big| =\frac{1}{k}|g(\sum_{i=1}^n t_i\omega_i)- \sum_{i=1}^n t_i\zeta_i|\leq \frac{1}{k}M<\epsilon.
$$
\end{proof}

It is also known that the extremal quasiconformal map for the Teichm\"uller distance is unique (see \cite{lehto}, Theorem 6.3). Interestingly, there are many extremal Lipschitz maps, at least in some cases.

\begin{prop}
\label{nonuniqueaffine}
There exists a pair of marked flat 2-tori with an infinite family of distinct homeomorphisms respecting the markings, all of which realize the extremal Lipschitz constant.  
\end{prop}
\begin{proof}
Let $S$ be the square $[0,1]\times[0,1]\subset\R^2$ and $T$ be the rectangle $[0,r]\times[0,1/r]$. These regions $S$ and $T$ represent fundamental domains for two flat tori. An extremal Lipschitz map is given by $(x,y)\mapsto(rx,y/r)$ with Lipschitz constant $r$. Fix $r>1$. Choose $\epsilon\in(-1/2,1/2)$ and $\delta$ such that
$$
\max\{0,\frac{1}{r}-\frac{r}{2}+\epsilon r\}<\delta < \min\{\frac{1}{r},\frac{r}{2}+\epsilon r\}.
$$
Define the map $F:S\to T$ by:
$$
F(x,y) =
\begin{cases}
\big(rx, \frac{1/r-\delta}{1/2-\epsilon} y\big) & y\leq 1/2 - \epsilon\\
\big(rx, \big(\frac{1}{r}-\delta\big) + \frac{y-(1/2-\epsilon)}{1/2+\epsilon}\delta & y\geq 1/2 - \epsilon
\end{cases}
$$
See the figure for an explanation of these values. 
\begin{figure}[htbp]
\label{nonuniquepic}
\includegraphics[width = 0.6\textwidth]{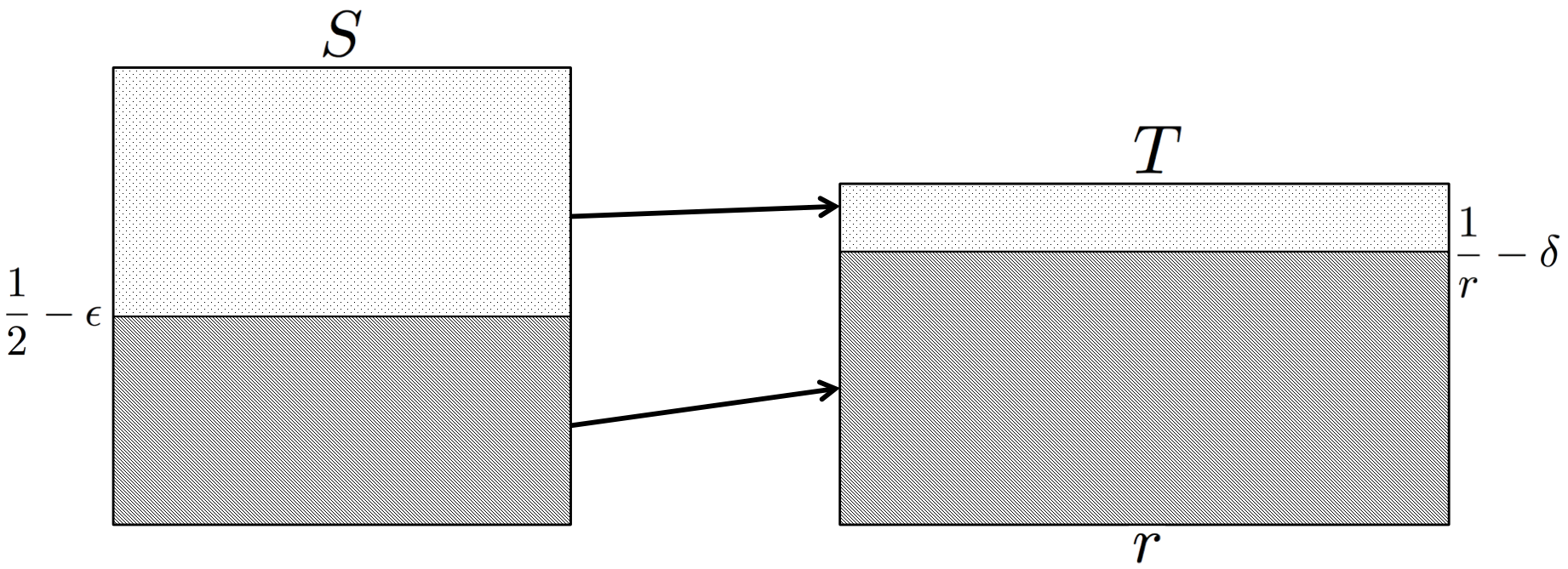}
\caption{The map $F$ sends the two portions of the square linearly to the two similarly-shaded portions of the rectangle.}
\end{figure}

This map is linear in the $x$-direction (the direction of maximum stretch), but only piecewise linear in the $y$-direction. The affine map occurs at $\epsilon=0$ and $\delta=(2r)^{-1}$. This map projects onto a homeomorphism of the corresponding tori since it respects the boundaries. The map $F$ is differentiable almost everywhere, and the total derivatives on the top and bottom halves of the domain are respectively given by:
$$
D_{\text{bottom}} = 
\begin{pmatrix}
r & 0\\
0 & \frac{1/r-\delta}{1/2-\epsilon}
\end{pmatrix},\ 
D_{\text{top}} = 
\begin{pmatrix}
r & 0 \\
0 & \frac{\delta}{1/2+\epsilon}
\end{pmatrix}
$$
With the above constraints on $\epsilon$ and $\delta$, one can see from $D_{\text{top}}$ and $D_{\text{bottom}}$ that the Lipschitz constant for $F$ is $r$, as desired. 
\end{proof}

In contrast to the case of the affine map, the inverses of the maps constructed in Proposition \ref{nonuniqueaffine} are not Lipschitz-extremal. The above construction generalizes easily to the case of higher dimensions.

\begin{cor}
\label{nonuniquentori}
There exists a pair of flat tori in any dimension $n\geq2$ with infinitely many homotopic homeomorphisms respecting the markings which all realize the extremal Lipschitz constant. 
\end{cor}
\begin{proof}

Let $S$ and $T$ be the two marked flat 2-tori from Proposition \ref{nonuniqueaffine} and let $S' = S\times(S^1)^{n-2}$ and $T' = T\times (S^1)^{n-2}$ with the product metrics, where each new copy of $S^1$ is isometric to a unit circle. An infinite extremal family is given by using the family from Proposition \ref{nonuniqueaffine} on the $S$ and $T$ components, and the identity on the remaining components. 
\end{proof}

\begin{rem}
 It is straightforward to generalize the above construction for any two rectangular tori, but it is unclear whether all pairs of tori admit many distinct Lipschitz-extremal maps, and if not, when they are unique.  
\end{rem}

\section{Thurston's metric for $n$-dimensional flat tori}
\label{ntorusthurston}

\begin{defn}
\emph{Thurston's metric} $d_{Th}$ on $\mathcal{T}(n)$ is defined as follows:
$$
d_{Th}([S,f],[S',f']) = \frac{1}{2}\log\inf_{\phi\in[f'^{-1}\circ f]}\mathcal{L}(\phi)
$$
$$
\mathcal{L}(\phi) = \sup_{x,y\in S,\ x\neq y}\frac{d_{S'}(\phi(x),\phi(y))}{d_S(x,y)}
$$
where the infimum is over all Lipschitz homeomorphisms homotopic to $f'^{-1}\circ f$. 
\end{defn}

This is identical to the definition for hyperbolic surfaces. Proposition 2.1 in \cite{thurston} gives a geometric proof that the Thurston metric is positive-definite for $\mathcal{T}(S_g)$, which works similarly for our case. 

\begin{prop}
\label{thurstonprop}
For all points $[S,f],[S'f']\in\mathcal{T}(n)$, we have 
$$
d_{Th}([S,f],[S',f'])\geq0,
$$
with equality only if $[S,f]=[S',f']$. 
\end{prop}
\begin{proof}
Suppose we have $[S,f],[S',f']$ such that $d_{Th}([S,f],[S',f'])\leq0$. Then by compactness there exists a homeomorphism $\phi:S\to S'$ in the appropriate homotopy class with realizing the extremal Lipschitz constant $L\leq1$. 

Under $\phi$ every sufficiently small ball of radius $r$ in the domain space is mapped to a subset of a ball of radius $\leq r$ in the target. However, both surfaces have unit volume. If we cover the domain space by a disjoint union of balls of full measure, one sees that each disk must map surjectively onto a disk of the same size. This procedure works for arbitrarily small balls, and so $\phi$ is an isometry. 
\end{proof}

Because composing Lipschitz maps with constants $L_1$ and $L_2$ gives a Lipschitz map with constant at most $L_1L_2$, the triangle inequality for $d_{Th}$ follows. Together with Proposition \ref{thurstonprop}, we have that $d_{Th}$ is a (possibly asymmetric) metric. We will need a quick classical fact before we can state a formula for $d_{Th}$. 

\begin{lem}
\label{operatornorm}
The Lipschitz constant of a linear map $M:\R^n\to\R^n$ is given by
$$
\max\{\sqrt{|\lambda|}\ :\ \lambda\text{ is an eigenvalue of }M^TM\}.
$$
\end{lem}
\begin{proof}
First, recall $\mathcal{L}(M)=||M||_{op}$, the operator norm of $M$:
$$
\mathcal{L}(M)=\sup_{x\neq y}\frac{||Mx-My||}{||x-y||} = \sup_{x\neq0}\frac{||Mx||}{||x||}=||M||_{op}.
$$
Since the operator norm of a diagonalizable matrix is the absolute value of the largest eigenvalue, using $||M^TM||_{op} = ||M||_{op}^2$ the result follows. 
\end{proof}

Next, we will derive a formula for easy computation using the structure of the symmetric space $\SL(n,\R)/\SO(n)$.

\begin{thm}
\label{lipschitzmetricfmla}
Let $Y,X$ be positive-definite symmetric matrices corresponding to points of $\mathcal{T}(n)$. Thurston's metric $d_{Th}$ on $\mathcal{T}(n) = \SL(n,\R)/\SO(n)$ is given by the following formula:
\begin{equation}
\label{dthfmla}
d_{Th}(Y,X) = \frac{1}{2}\max\{\log|\lambda|\ :\ \lambda\text{ is an eigenvalue of }XY^{-1}\}
\end{equation}
\end{thm}

\begin{proof}
Let $h\SO(n)$ and $g\SO(n)$ be points in $\SL(n,\R)/\SO(n)$ corresponding to $Y$ and $X$. The linear map between them is given by $gh^{-1}$, which by Proposition \ref{extremalaffine} is an extremal Lipschitz map. By Lemma \ref{operatornorm}, the Lipschitz constant is given by 
$$
\lambda_0 := \max\{|\lambda|\ :\ \lambda\text{ is an eigenvalue of }{(h^{-1})}^Tg^Tgh^{-1}\}
$$
Because $$XY^{-1} = g^Tgh^{-1}{(h^T)}^{-1} \sim {(h^{-1})}^Tg^Tgh^{-1}$$ are similar matrices, they have the same eigenvalues, and the result follows. 
\end{proof}

Notice that if $Y=I$, the absolute values in Equation \ref{dthfmla} are redundant since $X$ is positive-definite. 

\begin{cor}
The Thurston metric is $\SL(n,\R)$ invariant for the action on $\SL(n,\R)/\SO(n)\cong\mathcal{T}(n)$. 
\end{cor}
\begin{proof}
This is immediate from the formula and the definition of the action $g\cdot X = gXg^T$. 
\end{proof}

\begin{cor}
\label{dim2thurston}
The Thurston metric on $\mathcal{T}(2)$ is equal to the Riemannian symmetric metric on $\SL(2,\R)/\SO(2)$, and hence matches the Teichm\"uller metric and hyperbolic metric up to scaling.
\end{cor}
\begin{proof}
The distance formula for the Riemannian symmetric metric on $\SL(n,\R)/\SO(n)\cong\mathcal{P}_n$ is given by (see e.g. \cite{terrasvol2}, Theorem 1.1.1):
$$
d(X,Y) = \sqrt{\sum_i(\log\lambda_i)^2}
$$
where the sum is over the eigenvalues of $XY^{-1}$. In the case of $2\times 2$ matrices of determinant one, there are precisely two eigenvalues whose product is 1. Write the eigenvalue with absolute value at least 1 as $\lambda$. Then the formula becomes:
$$
d(X,Y) = \sqrt{(\log\lambda)^2 + (\log1/\lambda)^2} = \sqrt{2\log(\lambda)^2} = \sqrt{2}|\log\lambda|.
$$
But $\lambda$ is also the maximum eigenvalue of $XY^{-1}$, so up to a choice of scaling, these are the same metrics. 
\end{proof}

\begin{rem}
A proof of Corollary \ref{dim2thurston} is obtained in the unpublished work \cite{greenfield} by the present authors using an explicit computation of the Lipschitz distortion in a realization of the fundamental domains as parallelograms in $\C$. 
\end{rem}

\begin{rem}
Another proof of Corollary \ref{dim2thurston} is possible using work of Belkhirat-Papadopoulos-Troyanov \cite{bpt}, where the Thurston metric is defined on $\mathcal{T}(2)$, but $\mathcal{T}(2)$ is defined using a different normalization. A fixed curve is set to length 1 via the marking, as opposed to here, where we choose volume 1. Using the usual identification of $\mathcal{T}(2)\leftrightarrow\HH^2$, it is shown that the resulting Thurston metric, denoted here by $\hat\kappa$, can be computed by the following formula (\cite{bpt}, Theorem 3):
$$
\hat\kappa(\zeta,\zeta') = \log\sup_{\alpha\in\mathcal{S}}\bigg(\frac{\ell_{\zeta'}(\alpha)/\ell_{\zeta'}(\epsilon)}{\ell_{\zeta}(\alpha)/\ell_{\zeta}(\epsilon)}\bigg)=\log\bigg(\frac{|\zeta'-\overline{\zeta}|+|\zeta'-\zeta|}{|\zeta-\overline{\zeta}|}\bigg)
$$
where the supremum is over homotopy classes of closed curves, $\ell_{\zeta}(\alpha)$ is the length of $\alpha$ in the metric associated to $\zeta\in\HH^2$, and $\epsilon$ is the normalizing curve. In order to recover our $d_{Th}$, we normalize using $\sqrt{\Im\zeta}$, the volume. Using the identification $\mathcal{T}(2)\leftrightarrow\HH^2$ for $d_{Th}$, we obtain:
$$
d_{Th}(\zeta,\zeta') = \hat\kappa(\zeta,\zeta')+\log\bigg(\frac{\sqrt{\Im\zeta}}{\sqrt{\Im\zeta'}}\bigg)= \log\bigg(\frac{|\zeta'-\overline{\zeta}|+|\zeta'-\zeta|}{\sqrt{|\zeta-\overline{\zeta}||\zeta'-\overline{\zeta'}|}}\bigg) = \frac{1}{2}\log\bigg(\frac{|\zeta'-\overline{\zeta}|+|\zeta'-\zeta|}{|\zeta'-\overline{\zeta}|-|\zeta'-\zeta|}\bigg)
$$
where the last equality follows from Lemma 2 (an identity for complex numbers) from \cite{bpt}. This is exactly the Poincar\'e metric.
\end{rem}

Next, as in \cite{thurston}, we define another asymmetric metric, $\kappa$, on $\mathcal{T}(n)$. Let $\mathcal{S}(\mathbb{T}^n)$ denote the set of homotopy classes of essential closed curves on the $n$-torus. For $\alpha\in\mathcal{S}(\mathbb{T}^n)$ and $h$ a metric on $\mathbb{T}^n$, denote by $\ell_h(\alpha)$ the shortest length of any curve in the homotopy class $\alpha$. For the flat torus, while the curve realizing this length is not unique, the shortest length is well-defined. As above, let $[S,f],[S',f'] \in\mathcal{T}(n)$ with $h$ and $h'$ the corresponding unit-volume flat metrics on $\mathbb{T}^n$. Now, $\kappa$ is defined as: 
\begin{equation}
\label{kappadef}
\kappa([S,f],[S',f']) = \log \sup_{\alpha\in\mathcal{S}(T^n)}\bigg(\frac{l_{h'}(\alpha)}{l_{h}(\alpha)}\bigg)
\end{equation}

That is, $\kappa$ is a measure of the maximum stretch along a geodesic. As in \cite{thurston}, we show:

\begin{prop}
\label{kappalambdasame}
The two metrics $\kappa$ and $d_{Th}$ are equal on $\mathcal{T}(n)$.
\end{prop}
\begin{proof}
It is immediate that $$\kappa([S,f],[S',f'])\leq d_{Th}([S,f],[S',f'])$$ for all $[S,f],[S',f']\in\mathcal{T}(n)$, since the latter involves a supremum over all geodesic segments rather than only closed geodesics. For the opposite inequality, we will utilize a geometric argument. Let $\varphi:\R^n\to\R^n$ be the (lift of the) affine marking-preserving map between $S$ and $S'$. 

There exists a line $L$ containing the origin along which the maximal stretch of $\varphi$ is realized. If there are two lattice points on $L$, then the segment connecting them descends to a geodesic whose length is stretched by the Lipschitz constant, yielding $\kappa\geq d_{Th}$, and we are done. 

Suppose now 0 is the only lattice point on $L$. One can find a sequence of lattice points $p_n\in\Lambda$, $n=1,2,\ldots$ which approach $L$. By continuity, under $\varphi$ the corresponding sequence of closed geodesics will have stretch factors approaching the Lipschitz constant of the map $\varphi$. After taking the supremum of the stretches, we conclude $\kappa\geq d_{Th}$, as required. 
\end{proof}

\subsection*{The Finsler structure of the Thurston metric}

Finsler metrics are important in classical Teichm\"uller theory since both the Teichm\"uller metric and Thurston metric are Finsler but not Riemannian. Here, we will give a formula for the Finsler metric on $\mathcal{T}(n)$ associated to the Thurston metric $d_{Th}$. 

\begin{defn}
A \emph{Finsler metric} on a manifold M is a continuous function $$F:TM\to[0,\infty)$$ on the tangent bundle such that for each $p\in M$, the restriction $F|_{T_pM}:T_pM\to[0,\infty)$ is a norm (i.e. positive-definite, subadditive, linear under scaling by \emph{positive} scalars). 
\end{defn}

Our formula for the Finsler metric for $d_{Th}$ is very similar to the Finsler metric discussed in \cite{hilbertmetric} Theorem 3 (see also Section \ref{hilbertmetricsection} of this paper). Recall first that the tangent space of $\mathcal{T}(n)=\SL(n,\R)/\SO(n)$ at the identity is identified with the space of traceless symmetric matrices. One obtains any other tangent space by left translation via elements of $\SL(n,\R)$.

\begin{prop}
\label{finslermetricformula}
The Finsler structure on the tangent space at $Z\in\mathcal{T}(n)$ for the Thurston metric $d_{Th}$ is given by
$$
|X|_{Th(Z)} = \frac{1}{2}\max\{\lambda : \lambda\text{ is an eigenvalue of }XZ^{-1}\}
$$
where $X\in T_Z\mathcal{T}(n) \cong \fsl(n,\R)$. 
\end{prop}

\begin{proof}
It suffices to show the case of $Z=I$. First, note that this is always non-negative since the trace is zero, and defines a norm. Let $\gamma:[0,1]\to\SL(n,\R)/\SO(n)$ be a smooth path from $I$ to $A$. Since $A$ is symmetric, its operator norm coincides with the maximum eigenvalue, and so
$$
d_{Th}(I,A) = \frac{1}{2}\sup_{0\neq v\in\R^n}\log\frac{\langle\gamma(1)v,v\rangle}{\langle v,v\rangle}
$$
is the maximum eigenvalue. We then compute:
$$ 
d_{Th}(I,A)  =  \frac{1}{2}\sup_{0\neq v\in\R^n} \int_0^1\frac{d}{dt}\log\langle\gamma(t)v,v\rangle dt
= \frac{1}{2}\sup_{0\neq v\in\R^n}\int_0^1\frac{\langle \gamma'(t)v,v\rangle}{\langle\gamma(t)v,v\rangle}dt$$ $$\leq \frac{1}{2}\int_0^1\sup_{0\neq v\in\R^n}\frac{\langle \gamma'(t)v,v\rangle}{\langle \gamma(t)v,v\rangle}dt
= \frac{1}{2}\int_0^1|\gamma'(t)|_{Th(\gamma(t))}dt
$$
where the final equality follows because the supremum on the left-hand side yields the operator norm, which matches the Finsler norm inside the integral on the right-hand side. This is the Finsler length of $\gamma$. Thus $d_{Th}$ is bounded above by the Finsler distance of any path. 

Next, choose $X$ such that $e^X=A$, which exists because $A\in \mathcal{P}_n$. The Finsler length of the path $\gamma(t)=e^{tX}$ for $t\in[0,1]$ is computed as follows:
$$
\ell(\gamma) = \frac{1}{2}\int_0^1\sup_{0\neq v\in\R^n}\frac{\langle Xe^{tX}v,v\rangle}{\langle e^{tX}v,v\rangle}dt = \frac{1}{2}\int_0^1\sup_{0\neq v\in\R^n}\frac{d}{dt}\log\langle e^{tX}v,v\rangle$$ $$ = \frac{1}{2}\sup_{0\neq v\in \R^n}\frac{\langle Av,v\rangle}{\langle v,v\rangle} = d_{Th}(I,A)
$$
Thus the Thurston distance is realized the Finsler length of a path, as desired. 
\end{proof}

\begin{cor}
For $U,V\in\SL(n,\R)/\SO(n)$, if $e^X = UV^{-1}$, the path given by $t\mapsto e^{tX}V$ for $t\in[0,1]$ is a geodesic path from $V$ to $U$ with respect to $d_{Th}$. 
\end{cor}

\section{Teichm\"uller metric for higher-dimensional tori}
\label{teichhidim}

Here, we utilize quasiconformal maps for $\R^n$ from \cite{kqc} to define the Teichm\"uller metric on $\mathcal{T}(n)$ for $n\geq2$ and explore its properties.

\subsection{Definitions and useful facts on quasiconformal maps}

We will first state as concisely as possible the definition of $K$-quasiconformal maps between domains $D$ and $D'$ in $\R^n$ in the case of diffeomorphisms from Chapter 4 of \cite{kqc}. 

For a linear map $T:\R^n\to \R^m$, define the following:
$$
L(T) = \max_{|x|=1}|T(x)|,\ \ \ \ell(T)=\min_{|x|=1}|T(x)|.
$$
These are the maximal and minimal stretching of $T$, respectively. 

\begin{defn}
\label{kqcdefn}
Let $f:D\to D'$ be a diffeomorphism of domains in $\R^n$. Define the \emph{inner}, \emph{outer}, and \emph{maximal dilatations} respectively as follows:
$$
K_I(f) =\sup_{x\in D}\frac{|J_f(x)|}{\ell(f'(x))^n}
$$
$$
K_O(f) =\sup_{x\in D}\frac{L(f'(x))^n}{|J_f(x)|}
$$
$$
K(f) = \max(K_I(f),K_O(f))
$$
where $f'(x)$ is the total derivative of $f$ at $x\in D$ and $J_f$ is the Jacobian. The map $f$ is said to be $K$-quasiconformal if $K(f)\leq K<\infty$. 
\end{defn}

The above definition is local, so it applies immediately to flat tori by lifting any map to its universal cover. 

Next, we list a few basic properties of quasiconformal maps which will be essential to the definition of the Teichm\"uller metric and are direct analogs of the 2-dimensional case. These come from Lemma 6.1.1 and Theorem 6.8.4 of \cite{kqc}:
\begin{prop}
\label{kqcprops}
Let $f:D\to D'$ and $g:D'\to D''$ be quasiconformal homeomorphisms of domains in $\R^n$. Then the following hold:
\begin{enumerate}
\item $K(g\circ f)\leq K(g)K(f)$
\item $K(f)\geq1$ with equality if and only if $f$ is a M\"obius transformation
\item $K(f^{-1})=K(f)$
\end{enumerate}
\end{prop}

We will need one more property of quasiconformal maps in order to prove that the extremal quasiconformal constant is realized by the affine map. This is a very special case of Theorem 6.6.18 in \cite{kqc}. 

\begin{prop}
\label{kqcconverge}
Let $(f_k)_{k\in\N}:\R^n\to\R^n$ be a sequence of $K$-quasiconformal homeomorphisms. Suppose $f_k\to f$ locally uniformly. Then $f:\R^n\to\R^n$ is a $K$-quasiconformal homeomorphism as well. 
\end{prop}

We now prove the quasiconformal analog of Proposition \ref{extremalaffine}.

\begin{prop}
\label{kqcextremal}
The extremal quasiconformal constant for a homeomorphism between two flat $n$-tori in a specified homotopy class is given by the unique affine map.
\end{prop}
\begin{proof}
Recall the proof of Proposition \ref{extremalaffine}; in particular, recall the collection $\mathcal{F}$ of homeomorphisms $g:\R^n\to\R^n$ such that 
$$
g(0)=0,\ g(x+\sum_i^n m_i\omega_i) = g(x) + \sum_i^nm_i\zeta_i
$$
for all $x\in\R^n$. This is precisely the collection of lifts of marking-preserving homeomorphisms. Let $g\in\mathcal{F}$ be $K$-quasiconformal, and define $g_k(x)=g(kx)/k$ for $k=1,2,\ldots$. The maps $g_k$ are also $K$-quasiconformal since they are built from $g$ by pre- and post-composition with dilations. Further $g_k\in\mathcal{F}$, and the sequence of maps uniformly converges to the affine map. By Proposition \ref{kqcconverge}, the affine map has dilatation at most $K$. This holds for all $g\in\mathcal{F}$, so the result follows. 
\end{proof}

We are now ready to define the Teichm\"uller metric. 

\begin{defn} 
Let $[S,f],[S',f']\in\mathcal{T}(n)$. The \emph{Teichm\"uller metric} on $\mathcal{T}(n)$ is defined as:
$$
d_{Teich}([S,f],[S',f']) = \frac{1}{2n}\log\inf_{g\in[f'^{-1}\circ f]}K(g)
$$
where the infimum is taken over quasiconformal maps homotopic to $f'^{-1}\circ f$.
\end{defn}

\begin{prop}
The function $d_{Teich}$ above is a metric.
\end{prop}
\begin{proof}
Proposition \ref{kqcprops} (1) and (3) give symmetry and the triangle inequality, and (2) shows $d_{Teich}\geq0$. Now suppose $d_{Teich}([S,f],[S',f'])=0$. Then there exists a 1-quasiconformal map $g:S\to S'$ preserving the marking. By Proposition \ref{kqcprops} (2), it must be a M\"obius transformation. Since it preserves the marking, it must be orientation-preserving and not include inversions in spheres. Thus it is generated by an even number of reflections over hyperplanes, so it is (the quotient of) an orientation-preserving isometry of $\R^n$. We conclude $[S,f]=[S',f']$. 
\end{proof}

Next, we exhibit a significant departure from Teichm\"uller spaces of hyperbolic surfaces. 

\begin{thm}
\label{symofthurst}
For all $[S,f],[S',f']\in\mathcal{T}(n)$, we have:
$$
d_{Teich}([S,f],[S',f'])= \max(d_{Th}([S,f],[S',f']),d_{Th}([S',f'],[S,f]))
$$
\end{thm}

\begin{proof}
Recall from Corollary \ref{kqcextremal} that the extremal quasiconformal constant between two marked flat $n$-tori is realized by the unique affine map. The Jacobian of an affine map is equal to its determinant, which must be 1, since it must be volume-preserving. Definition \ref{kqcdefn} then gives  
$$
K(g) = \max\bigg(\sup_{x\in\R^n}L(g'(x))^n,\sup_{x\in\R^n}\frac{1}{\ell(g'(x))^n}\bigg).
$$
But $g$ is affine, so $L(g'(x))$ is the Lipschitz constant of $g$, and $\ell(g'(x))^{-1}$ is the Lipschitz constant of the inverse map. 
\end{proof}

\begin{cor}
\label{teichfmla}
The Teichm\"uller metric on $\mathcal{T}(n)$ is given by:
$$
d_{Teich}(X,Y) = \frac{1}{2}\max\{\big|\log|\lambda|\big|\ :\ \lambda\text{ is an eigenvalue of }XY^{-1}\}
$$
\end{cor}
\begin{proof}
This is precisely the symmetrization of the formula from Theorem \ref{lipschitzmetricfmla}, since the eigenvalues of $YX^{-1}$ are the reciprocals of the eigenvalues of $XY^{-1}$.
\end{proof}

\section{The Hilbert metric on $\SL(n,\R)/\SO(n)$}
\label{hilbertmetricsection}

Liverani and Wojtkowski \cite{hilbertmetric} defined a generalization of Hilbert's projective metric for the symmetric space $X=\SL(n,\R)/\SO(n)$. Their metric $s$ arises naturally during the study of the symplectic geometry of $\R^n\times\R^n$, and measures the distance between pairs of Lagrangian subspaces. An explicit formula for the Finsler metric on the tangent space $T_ZX$ at a point $Z\in X$ associated to their Hilbert metric is also computed, along with examples of geodesics. 

Consider the standard symplectic vector space $\R^n\times\R^n$, where the symplectic form is given by:
$$
\omega((x,y),(w,z)) = \langle x,z\rangle_{\R^n} - \langle w,y\rangle_{\R^n}
$$
A subpsace $V$ of $(\R^n\times\R^n,\omega)$ is called \emph{Lagrangian} if it is a maximal subspace such that $\omega|_V\equiv0$. These subspaces must be $n$-dimensional. A Lagrangian subspace is \emph{positive} if it is the graph of a positive-definite symmetric linear map $U:R^n\to\R^n$. The collection of positive Lagrangian subspaces is parametrized by the space $\mathcal{P}_n$. 

The metric $s$ is defined as the supremum of the symplectic angle between vectors in two positive Lagrangian subspaces. A useful result is the following formula.

\begin{prop}[Proposition 5, Theorem 3, \cite{hilbertmetric}]
For two positive Lagrangian subspaces defined by $U,W:\R^n\to\R^n$, $s$ is given by
\begin{equation}
\label{lwformula}
s(U,W)  =\max\bigg\{\frac{\big|\log|\lambda|\big|}{2}\ :\ \lambda\text{ is an eigenvalue of }UW^{-1}\bigg\}.
\end{equation}
The Finsler norm $|A|_Z$ for $A\in T_ZX$ is given by
\begin{equation}
\label{finslerequation}
|A|_Z = \frac{1}{2}\max\{|\lambda|\ :\ \lambda\text{ is an eigenvalue of }AZ^{-1}\}.
\end{equation}
and the paths $t\mapsto e^{tX}$ for $t\in[0,1]$ are geodesic paths.
\end{prop}

Notice that Equation \ref{lwformula} matches the formula in Corollary \ref{teichfmla}, so we conclude:

\begin{prop}
\label{hilbertthurstonmetric}
By the identification $\mathcal{T}(n)\leftrightarrow\SL(n,\R)/\SO(n)$, $d_{Teich}$ is equal to the Hilbert projective metric, and $d_{Teich}$ is a Finsler metric with norm defined by Equation \ref{finslerequation}. The paths $t\mapsto e^{tX}$ for $t\in[0,1]$ are geodesics. 
\end{prop}

The significance of Proposition \ref{hilbertthurstonmetric} is that the same metric $d_{Teich}$ on $\mathcal{T}(n)$ arises in a natural way in a very different context. This provides further evidence of the usefulness and richness of the study of this Finsler metric on $\SL(n,\R)/\SO(n)$.

\begin{rem}
The Hilbert metric, defined on open convex subsets $C\subseteq\R^n$ not containing a line, is based on the cross-ratio of two points $a,b$ and the points where the line $\overline{ab}$ meets the boundary $\partial C$. When $C$ is the positive orthant of $\R^n$, one obtains a Finsler metric with many properties similar to the metric $s$. 
\end{rem}

\section{The Weil-Petersson metric}
\label{wpmetricsection}
In this section, we will define the Weil-Petersson metric on $\mathcal{T}(n)$. Fischer-Tromba \cite{fischertromba} show the classical Weil-Petersson metric is recovered using a $L^2$-pairing between metrics on hyperbolic surfaces. In \cite{yamada}, Yamada gives an exposition of this approach, including a definition of the Weil-Petersson metric for the Teichm\"uller space of the flat 2-torus. We will follow Yamada's presentation and explain how this quickly generalizes to the case of flat tori in all dimensions.

Write $K=\SO(n)$, $G=\SL(n,\R)$. Recall first the tangent space of $G/K$ at the basepoint $eK$ is the vector space of $n\times n$ symmetric matrices of trace 0. The $\SL(n,\R)$-invariant metric at this point is defined by:
$$
\langle X,Y\rangle_{eK} = \tr(XY).
$$
By translation, at other points $g K\in X$ for $g\in\SL(n,\R)$ the metric is given by
\begin{equation}
\label{symmetricmetric}
\langle X,Y\rangle_{gK} = \tr(g^{-1}Xg^{-1}Y).
\end{equation}
Now, recall that $\mathcal{T}(n)\cong\mathcal{P}_n$ is also the space of unit-volume flat metrics on $\mathbb{T}^n$. 

The tangent space to the set of Riemannian metrics on a manifold is naturally the space of smooth symmetric $(0,2)$-tensors (\cite{yamada}, \S 3). There is a natural $L^2$ pairing $\langle\langle,\rangle\rangle_{L^2(g)}$ at a metric $g$ defined by:

\begin{equation}
\label{l2pairing}
\langle\langle h_1,h_2\rangle\rangle_{L^2(g)} = \int_M\langle h_1(x),h_2(x)\rangle_{g(x)}d\mu_g(x)
\end{equation}
using the volume form $d\mu_g$ of $g$. Using local coordinates $g^{ij}$ for $g$ and $(h_k)_{lm}$ for $h_k$, $k=1,2$, we can rewrite the integrand as:
$$
\langle h_1(x),h_2(x)\rangle_{g(x)} = \sum_{1\leq i,j,k,l\leq2}g^{ij}g^{kl}(h_1)_{ik}(h_2)_{jl}=\text{Tr}(g^{-1}h_1g^{-1}h_2).$$

In \S3.2 of \cite{yamada}, two conditions are imposed on the deformations of a metric in order to ensure that each tensor $h$ is tangent to the Teichm\"uller space and not merely the space of all possible metrics: (1) the deformations must be $L^2$-perpendicular to the action of the identity component of the diffeomorphism group $\text{Diff}_0(M)$, and (2) the deformations must preserve curvature. It is shown there that these conditions are equivalent to being divergence-free and trace-free. 

Finally, we arrive at the definition of the Weil-Petersson metric on Teichm\"uller space with the viewpoint of deformations of Riemannian metrics. 
\begin{defn}[\cite{fischertromba}, Theorem 0.8]
The $L^2$-pairing in Equation \ref{l2pairing} restricted to the trace-free, divergence-free tensors is called the \emph{Weil-Petersson metric}.
\end{defn}

We apply the above definitions to $\mathcal{T}(n)$. Deformations of flat metrics which remain in the Teichm\"uller space define a subspace of all $(0,2)$-tensors. Maintaining unit volume restricts to traceless tensors, while the restriction to flat metrics implies the tensors are constant $\R^n$-coordinates. These are trace-free and divergence-free tensors. Thus the integrand in Equation \ref{l2pairing} is constant and given globally by the local coordinates. The volume of each metric is 1, so the $L^2$-pairing simplifies to:
$$
\langle\langle h_1,h_2\rangle\rangle_{L^2(g)} = \text{Tr}(g^{-1}h_1g^{-1}h_2).
$$
This matches the symmetric metric in Equation \ref{symmetricmetric}. We now have for all $n\geq2$:

\begin{prop}
\label{wpmetric}
The Teichm\"uller space $\mathcal{T}(n)$ with the Weil-Petersson metric is isometric to the $\SL(n,\R)/\SO(n)$ with the $\SL(n,\R)$-invariant Riemannian metric.
\end{prop}

\section{Symmetric spaces and compactifications}
\label{symmetricbackground}

We briefly review some relevant ideas about symmetric spaces and compactifications. The main references are \cite{hall}, \cite{bridson}, \cite{borel}, \cite{gjt}, and \cite{hsw}. In the following, let $G=\SL(n,\R)$, $K=\SO(n)$, and $X=G/K$.

\subsection*{Lie theory and symmetric spaces}
\label{liebackground}

The Lie algebra of $G$ is $\mathfrak{g}=\fsl(n,\R)$ consisting of traceless matrices, which decomposes as 
$$
\mathfrak{g} = \mathfrak{k}\oplus\mathfrak{p}
$$
where $\mathfrak{k}$ is the Lie algebra of $K$, consisting of traceless anti-symmetric matrices, and $\mathfrak{p}$ consists of traceless symmetric matrices.

Fix a \emph{Cartan subalgebra} $\mathfrak{a}\subseteq\mathfrak{p}$ consisting of traceless diagonal matrices. The dimension of $\mathfrak{a}$ is the \emph{rank} of $G$ and of $X$. Here, the rank is $r=n-1$. Denote $A=\exp(\mathfrak{a})$, the subgroup of $G$ corresponding to the subalgebra $\mathfrak{a}$. A totally geodesic copy of $\R^r$ embedded in the symmetric space $X$ is called a \emph{maximal flat}. 

We next recall a few important examples of representations of $G$ and $\mathfrak{g}$. 

\begin{exmp}
The \emph{standard representation} of $G$ is the inclusion $$\Pi:\SL(n,\R)\hookrightarrow\text{GL}(n,\C).$$ This is a faithful representation. The standard representation of $\mathfrak{g}$ is the inclusion $$\pi:\fsl(n,\R)\hookrightarrow\text{M}_n(\C).$$ Composing $\Pi$ with the quotient map $\text{GL}(n,\C)\to\text{PGL}(n,\C)$ defines a projective faithful representation $$\Pi_P:\SL(n,\R)\to\text{PGL}(n,\C).$$ 
\end{exmp}

\begin{exmp}
The \emph{adjoint representation} of the Lie algebra $\mathfrak{g}$ is defined by
$$
\text{Ad}:\mathfrak{g}\to\text{M}_n(\C),\ A\mapsto[A,\cdot]\text{ for }A\in\mathfrak{g}
$$
\end{exmp}

The \emph{dual} of a representation $\Pi$ of $G$ is the representation $\Pi^*$ defined by $$\Pi^*(g) = \tau(g^{-1})^T$$ where $A^T$ is the transpose of $A$. The dual of a representation $\pi$ of a Lie algebra is defined by $$\pi^*(A) = -\pi(A)^T.$$

The direct sum of two representations $\tau_1:G\to\text{GL}(n,\C)$ and $\tau_2:G\to\text{GL}(m,\C)$ is the representation $\tau_1\oplus\tau_2:G\to\text{GL}(n+m,\C)$ with the diagonal action.

We next recall weights and roots associated to $\mathfrak{a}$. A natural inner product on $\mathfrak{a}$ is given by
$$
\langle A,B\rangle = \tr(\overline{A}^{T}B).
$$ 
This inner product identifies $\mathfrak{a}$ with the dual space $\mathfrak{a}^*$. Let $\pi$ be a nonzero representation of $\mathfrak{g}$ acting on $\R^m$. We say $\mu\in\mathfrak{a}$ is a \emph{weight} for $\pi$ if there exists a nonzero $v\in \R^m$ such that 
\begin{equation}
\label{weightdef}
\pi(H)\cdot v = \langle \mu,H\rangle v
\end{equation}
for all $H\in\mathfrak{a}$. The weight space of $\mu$, denoted $V_{\mu}$, is the subspace of all $v\in\R^m$ for which Equation \ref{weightdef} holds. Each representation of a Lie group has an associated representation of its Lie algebra. The weights of a Lie group representation are defined to be the weights of the associated Lie algebra representation. 

\begin{exmp}
Let $\pi$ be the standard representation for $\fsl(n,\R)$. Then the weights are given by the standard basis $e_i$, so $\langle e_i,\cdot\rangle$ returns the $i$th diagonal element of a matrix, and the weight space for $e_i$ is the line $\{\lambda e_i\ :\ \lambda\in\R\}$.

Let $\pi^*$ be the dual of the standard representation. Then the weights are $-e_i$ with corresponding weight spaces generated by $e_i$ after identifying $\R^n$ with its dual. 
\end{exmp}

Let $\Pi_1\oplus\Pi_2$ be a direct sum of two representations acting on $V\oplus W$, and let $$\mathcal{W}_1=\{\mu_i\ :\ i=1,\ldots,n\}\text{  and  }\mathcal{W}_2=\{\nu_j\ :\ j=1,\ldots,m\}$$ be the weights of $\Pi_1$ and $\Pi_2$ respectively, with corresponding weight spaces $V_i\subseteq V$ and $W_j\subseteq W$. Then the weights of $\Pi_1\oplus\Pi_2$ are $\mathcal{W}_1\cup\mathcal{W}_2$ with weight spaces $V_i\oplus\{0\}$ and $\{0\}\oplus W_j$ when $\mu_i\notin\mathcal{W}_2$ and $\nu_j\notin\mathcal{W}_1$. If some $\mu_i=\nu_j$, then its (common) weight space is $V_i\oplus W_j$.

The set of \emph{roots} of $\mathfrak{g}$ relative to $\mathfrak{a}$, denoted $\Sigma$ are the weights of the adjoint representation. A set $\Delta$ of \emph{simple roots} is a basis of $\mathfrak{a}$ made up of roots such that any root for $\mathfrak{a}$ can be expressed as an integer linear combination of elements of $\Delta$ where all coefficients are non-positive or non-negative. 

\begin{exmp}
A set of simple roots for $\fsl(n,\R)$ with the Cartan subalgebra $\mathfrak{a}$ defined above is given by $$\alpha_1=(1,-1,0,\ldots,0),\ \alpha_2=(0,1,-1,0,\ldots,0),\ldots,\ \alpha_{n-1}=(0,\ldots,0,1,-1).$$ The root space for $\alpha_j$ is spanned by the matrix $E^{j,j+1}$ which has a 1 in the $(j,j+1)$ spot and 0 elsewhere. 
\end{exmp}

Given a representation of $\mathfrak{g}$, a choice of simple roots endows the set of weights with a partial ordering (\S8.8 in \cite{hall}). If $\{\alpha_1,\ldots,\alpha_n\}$ is the set of simple roots of $\mathfrak{g}$ and $\lambda_1,\lambda_2$ are weights of a representation, we say $\lambda_2\succeq\lambda_1$ if there exist non-negative real numbers $c_1,\ldots,c_n$ such that 
$$
\lambda_2-\lambda_1=c_1\alpha_1+\cdots+c_n\alpha_n.
$$
It is a fundamental result (Theorems 9.4 and 9.5 in \cite{hall}) that irreducible, finite-dimensional representations of semisimple Lie algebras are classified by their highest weights (which always exist). 

To each root $\alpha$ of $\mathfrak{g}$ is associated a hyperplane $P_{\alpha} = \ker(\langle\alpha,\cdot\rangle)$. The complement of these hyperplanes, $\mathfrak{a}\setminus \cup_{\alpha\in\Sigma}P_{\alpha}$, is a set of simplicial complexes, each connected component of which is called a \emph{Weyl chamber}. A choice of a set of simple roots corresponds to distinguishing a \emph{positive} Weyl chamber.

Now, we define a special type of Finsler metric built from Minkowski norms which plays a major role in the theory of compactifications of symmetric spaces.

\begin{defn}
A \emph{polyhedral Finsler metric} on a symmetric space is a Finsler metric such that for each tangent space, the induced unit ball is a polytope.
\end{defn}

\begin{thm}[\cite{planche}, Theorem 6.2.1]
\label{finslerballs}
The following are in natural bijection:
\begin{enumerate}
\item the $W$-invariant convex closed balls in $\mathfrak{a}$
\item the $\text{Ad}(K)$-invariant convex closed balls of $\mathfrak{p}$
\item the $G$-invariant Finsler metrics on $X=G/K$
\end{enumerate}
\end{thm}

The idea of this theorem is that, given a Finsler metric on a maximal flat $F$ of $G/K$, if it is invariant under the Weyl group action, it can be extended to all of $G/K$ by enforcing $G$-invariance. This defines a $G$-invariant Finsler metric.

\subsection*{Compactifications}

Let $X$ be a locally compact space. A \emph{compactification} of $X$ is a pair $(\overline{X},i)$ where $\overline{X}$ is a compact space and $i:X\to \overline{X}$ is a dense topological embedding. If $(\overline{X}_1,i_1)$ and $(\overline{X}_2,i_2)$ are compactifications of $X$, we say they are \emph{isomorphic} if there exists a homeomorphism $\phi:\overline{X}_1\to \overline{X}_2$ such that $\phi\circ i_1=i_2$. If $\phi$ is only continuous, then it is necessarily surjective, and $(\overline{X}_1,i_1)$ is said to \emph{dominate} $(\overline{X}_2,i_2)$. Domination puts a partial order on the set of compactifications. 

In the case of symmetric spaces $X=G/K$, we are also interested in compactifications that admit a continuous $G$-action. The relations of \emph{$G$-isomorphism} and \emph{$G$-compactification} are extensions of the above definitions with the added condition of equivariance under the $G$ action.

Horofunction compactifications were introduced by Gromov \cite{gromov}. Let $(X,d)$ be a (possibly asymmetric) proper metric space with $C(X)$ the set of continuous real-valued functions on $X$ endowed with the $C^0$ topology. Denote by $\tilde{C}(X)$ the quotient of $C(X)$ by constant functions. We embed $X$ into $\tilde{C}(X)$ as follows:
$$
\psi:X\to\tilde{C}(X),\ z\mapsto[\psi_z]\text{ where }\psi_z(x)=d(x,z).
$$

\begin{defn}
The \emph{horofunction compactification} $X\cup\partial_{hor}X$ of $X$ is the topological closure of the image of $\psi$:
$$
\overline{X}^{hor}:= \text{cl}\{[\psi_z]|z\in X\}\subseteq \tilde{C}(X)
$$
\end{defn}

It is known that the horofunction compactification of a non-positively curved, complete, simply-connected Riemannian symmetric space $G/K$ with its $G$-invariant metric is naturally isomorphic to its visual compactification. This holds more generally for CAT(0) spaces (Theorem 8.13, \S II.8 in \cite{bridson}).

Next, we briefly review Satake compactifications of symmetric spaces, first defined in \cite{satake}. See also Chapter IV of \cite{gjt} and Chapter I.4 of \cite{borel}, and \cite{hsw} \S5.1 for generalized Satake compactifications. 

Let $X=G/K$ be a symmetric space associated to a semisimple Lie group $G$ with maximal compact subgroup $K$. Let $\tau:G\to\text{PSL}(m,\C)$ be an irreducible projective faithful representation such that $\tau(K)\subseteq \text{PSU}(m)$. This induces a map $$\tau_X:X\to\mathbb{P}(\mathcal{H}_n)$$ where $\mathbb{P}(\mathcal{H}_n)$ is the projective space of Hermitian matrices, defined by
$$
\tau_X(gK) = \tau(g)\overline{\tau(g)}^{T}.
$$
This is a topological embedding (Lemma 4.36 in \cite{gjt}). 

\begin{defn}
The \emph{Satake compactification of $X$ associated to $\tau$} is the closure of $\tau_X(X)$ in $\mathbb{P}(\mathcal{H}_n)$ and is denoted by $\overline{X}_{\tau}^S$. 
\end{defn}

Two Satake compactifications are $G$-isomorphic if and only if the highest weights of their representations lie in the same Weyl chamber face, so there are only finitely many different $G$-isomorphism types (Chapter IV, \cite{gjt}). 

\begin{defn}
The \emph{maximal Satake compactification} of a symmetric space is a Satake compactification whose highest weight lies in the interior of the positive Weyl chamber. A \emph{minimal Satake compactification} of a symmetric space is a Satake compactification whose highest weight lies in an edge of the Weyl chamber. 
\end{defn}
It is known that there is a unique (up to $G$-isomorphism) maximal Satake compactification which dominates all other Satake compactifications, and many minimal Satake compactifications. For $\SL(n,\R)$, it is known that the standard representation induces a minimal Satake compactification \cite[Proposition I.4.35]{borel}. 

We will also need \emph{generalized Satake compactifications}, the definition of which differs only in that the assumption that $\tau$ is irreducible is dropped.

In \cite{hsw}, Haettel, Schilling, Walsh, and Wienhard related generalized Satake compactifications of a symmetric space to horofunction compactifications of polyhedral Finsler metrics.

\begin{thm}[\cite{hsw} Theorem 5.5]
\label{hswmainthm}
Let $\tau:G\to\PSL(n,\C)$ be a projective faithful representation, and $X=G/K$ be the associated symmetric space, where $X$ is of non-compact type. Let $\mu_1,\ldots,\mu_k$ be the weights of $\tau$. Let $d$ be the polyhedral Finsler metric whose unit ball in a Cartan subalgebra is 
$$
B = -D^{\circ} = -\text{conv}(\mu_1,\ldots,\mu_k).
$$
where $\text{conv}$ is the convex hull. Then $\overline{X}^S_{\tau}$ is $G$-isomorphic to $\overline{X}^{hor}_d$. 
\end{thm}

\begin{exmp}
The horofunction compactification of $X=\SL(n,\R)/\SO(n)$ with respect to the standard $\SL(n,\R)$-invariant Riemannian metric is not isomorphic to a generalized Satake compactification because the unit ball in a flat is a Euclidean ball, which is not the convex hull of finitely many points. 
\end{exmp}

\section{Horofunction and Satake compactifications}
\label{compactifications}

In this section, we will describe horofunction compactifications of $\mathcal{T}(n)$ with the Thurston and Teichm\"uller metrics defined in Sections \ref{ntorusthurston} and \ref{teichhidim}.

\subsection*{The Thurston Metric}

Recall that the standard representation of $\SL(n,\R)$ induces a minimal Satake compactification of $\mathcal{T}(n)$. It has the following metric realization. 
\begin{prop}
\label{thurstonhoroboundary}
The following compactifications are $G$-isomorphic:
$$
\overline{\mathcal{T}(n)}^{hor}_{d_{Th}} \cong_G \overline{\mathcal{T}(n)}^S_{\pi}
$$
where $\Pi$ is the standard representation of $G=\SL(n,\R)$. 
\end{prop}

\begin{proof}
The weights of the standard representation are simply the standard basis for $e_i$, $i=1,\ldots,n$, for $\R^n$. Projecting them onto the hyperplane in $\R^n$ corresponding to $\mathfrak{a}$, the set of weights may be given by:
$$
\mu_i:=e_i - \sum_{j=1}^n\frac{1}{n}e_j,\ i=1,\ldots,n.
$$ 
Following \cite{hsw}, consider the convex hull $D:=\text{conv}(\mu_1,\ldots,\mu_n).$ This lies within the codimension 1 hyperplane $\sum_ix_i=0$ in $\R^n$. In order to utilize Theorem \ref{hswmainthm}, we now compute the negative of the dual polytope of $D$. If $\{a_1,\ldots,a_k\}\subseteq \R^n$ are the vertices of a convex polytope, then the dual polytope is given by:
$$
\{y\in \R^n : \langle a_i, y\rangle \geq-1\ \forall i\}.
$$
The extremal points are those where equality holds. By symmetry, the $\mu_i$'s are extremal points for the convex hull $D$, and the dual must live in the same hyperplane, so this becomes:
$$
B_0:= -D^{\circ} = -\{(y_1,\ldots,y_n)\in \R^n : y_1+\cdots+y_n=0,\ y_i - \frac{1}{n}\sum_jy_j\geq -1\ \forall i\} $$ $$= \{(y_1,\ldots,y_n)\in \R^n : y_1+\cdots+y_n=0,\ y_i \leq 1\ \forall i\} 
$$
By Theorem \ref{hswmainthm}, this is a unit ball for a polyhedral Finsler metric whose horofunction compactification is the Satake compactification of the standard representation. 

To complete the proof, we compute the unit ball of the Finsler metric $d_{Th}$ in the Cartan subalgebra. Using the formula in Proposition \ref{finslermetricformula}, this is relatively straightforward:
$$
B = \{(y_1,\ldots,y_n)\in\R^n : y_1+\cdots+y_n=0,\ y_i \leq 2\ \forall i\}
$$
Because $B_0=B$ up to scaling, we are done.
\end{proof}

\subsection*{The Teichm\"uller metric} We have a similar result for $d_{Teich}$.

\begin{prop}
\label{teichmullerboundary}
Let $\Pi$ be the standard representation of $G=\SL(n,\R)$. Then the following compactifications are $G$-isomorphic:
$$
\overline{\mathcal{T}(n)}^{hor}_{d_{Teich}} \cong_G \overline{\mathcal{T}(n)}^S_{\Pi\oplus\Pi^*}
$$
\end{prop}

\begin{proof}
Consider the faithful representation $$\Pi\oplus\Pi^*:\SL(n,\R)\hookrightarrow \SL(2n,\C),$$ using the standard and dual representations as a block diagonal acting on the direct sum of the vector spaces.  

The collection of weights, viewed as elements of $\R^n$, is the union of the weights for the standard and dual representations. We project them onto the hyperplane $P\subseteq\R^n$ defined by $\sum_iy_i=0$ to obtain the weights in $\mathfrak{a}$. After projection, two of the weights are given by $$a_1:=\bigg(\frac{n-1}{n}, -\frac{1}{n}, \ldots, -\frac{1}{n}\bigg),\ b_1=\bigg(-\frac{n-1}{n}, \frac{1}{n},\ldots, \frac{1}{n}\bigg),$$ and the others are similar, with $\pm(1-1/n)$ in the $i$th component and $\mp1/n$ in the remaining components. We consider the convex hull $D$ of these points. This defines a polyhedron in $\R^n$, of which we compute the negative of the dual. 

\begin{lem}
\label{dualball}
The negative of the dual to the polyhedron $D=\conv(a_1,\ldots,a_n,b_1,\ldots,b_n)$ is given by:
$$
-D^{\circ} = \{(y_1,\ldots,y_n)\in\R^n\ :\ \sum_iy_i=0,\ |y_i|\leq1\ \forall i=1,\ldots,n\}
$$
\end{lem}
We prove this lemma below. Now, using Equation \ref{finslerequation} for the Finsler metric associated of $d_{Teich}$, we see that the ball $-D^{\circ}$ is, up to a choice of scaling, the same as the unit ball for the Teichm\"uller metric. Theorem \ref{hswmainthm} completes the proof. 
\end{proof}

\begin{proof}[Proof of Lemma \ref{dualball}]
Since all points $a_1,\ldots, b_n$ lie in the hyperplane $\sum_iy_i=0$, the dual polyhedron must as well. Now, choose some $i\in\{1,\ldots,n\}$ and consider the condition $\langle (y_1,\ldots,y_n)|a_i\rangle\geq-1$. Expanding, this becomes:
$$
-\frac{1}{n}(y_1+\cdots+y_n)+y_i\geq-1
$$
But since $\sum_iy_i=0$, this simplifies to $y_i\geq-1$. For $b_i$, we obtain $1\geq y_i$. 
\end{proof}

\section{The Thurston compactification of $\mathcal{T}(n)$}
\label{thurstoncompact}

Inspired by Thurston's compactification for Teichm\"uller spaces of hyperbolic surfaces using projective measured laminations on the underlying surfaces, we define a natural Thurston-type compactification of $\mathcal{T}(n)$. It is closely related to T. Haettel's compactification of $\SL(n,\R)/\SO(n)$ built from the closure of a projective embedding into $\mathbb{P}(\R^{\Z^n}_+)$ in \cite{haettel}, but we provide a new construction utilizing a geometric interpretation of quadratic forms.

Recall the Satake compactification of $\SL(n,\R)/\SO(n)$ with respect to the standard representation of $\SL(n,\R)$, whose boundary points correspond to projective classes of positive-semidefinite matrices. After relating this compactification to the Thurston compactification, we have a geometric interpretation of the Satake compactification $\overline{\mathcal{T}(n)}^S_{\pi}$. Recall (see \cite{fathi}) that a measured foliation on a surface is a (singular) foliation with an arc measure in the transverse direction that is invariant under holonomy (translations along leaves). 

\begin{defn}
\label{mff}
A \emph{measured flat foliation} on $\R^n/\Z^n$ is a non-singular measured foliation $(F,\mu)$ with the following requirements:
\begin{itemize}
\item The leaves of $F$ are given by parallel hyperplanes.
\item The measure $\mu$ is invariant under isometries of the torus.
\item In the lift to $\R^n$, if $V_0$ is the leaf containing the origin, then there exists an orthogonal decomposition $$V_0^{\perp} = V_1\oplus\cdots\oplus V_k$$ and positive constants $\lambda_i$, $i=1,\ldots,k$, such that the lift of an arc $\gamma$ contained in subspace $V_i$ has measure $\mu(\gamma) = \lambda_i \ell_I(\gamma)$, where $\ell_I$ is the Euclidean length.
\end{itemize}
\end{defn}

This is a higher-dimensional analog of measured foliations for surfaces where the leaves are totally geodesic submanifolds. Invariance under isometries implies that we may assume any arc to be measured has a lift that begins at the origin in $\R^n$. There is an obvious action by $\R^+$ on the set of measured flat foliations by scaling the measure. Denote the set of projective classes of measured flat foliations by $\mathcal{PMFF}$.

\begin{lem}
\label{measuredflat}
The collection $\mathcal{PMFF}$ is in a natural one-to-one correspondence with the boundary of the minimal Satake compactification associated to the standard representation. 
\end{lem}

\begin{proof}

Let $Q$ be a matrix representative of the class $[Q]\in\partial\overline{\mathcal{T}(n)}^S_{\pi}$. Define the leaves of a foliation of $\R^n$ by all parallel translations of $\ker(Q)$. This descends to the quotient $\mathbb{R}^n/\Z^n$. Arc length with respect to $Q$ defines a transverse measure, which for an arc $\gamma:[0,1]\to\R^n/\Z^n$ is given by $$\ell_Q(\gamma) = \int_0^1\sqrt{\langle \gamma'(t) Q,\gamma'(t)\rangle} dt.$$ Because the quadratic form $Q$ is constant across $\R^n/\Z^n$ and diagonalizable, the measure satisfies the conditions in Definition \ref{mff}.  

In this way, $Q$ endows $\R^n/\Z^n$ with a measured foliation. Taking the projective class gives us the projective measured flat foliation associated to $[Q]$. 

Conversely, given $(F,[\mu])\in\mathcal{PMFF}$, we can obtain the associated $[Q]\in \partial\overline{\mathcal{T}(n)}_{\pi}^S$ as follows. Take any representative $(F,\mu)$ of the projective class. Then: 
\begin{enumerate}
\item Lift the measured foliation to $\R^n$
\item Let $v_1,\ldots,v_m$ be an orthonormal basis of the subspace $V_0$ spanned by the leaf through the origin
\item For each subspace $V_j$ in the direct sum $V_0^{\perp} = V_1\oplus\cdots\oplus V_k$ from Definition \ref{mff}, choose an orthonormal basis. Label these vectors $v_{m+1},\ldots,v_n$
\item Let $\lambda_i$ be the measure of a straight line segment of Euclidean length 1 extending from the origin in the direction of $v_i$ for $i=1,\ldots,n$
\item Let $P$ be the matrix whose columns are $v_i$ for $i=1,\ldots,n$ and let $D$ be the diagonal matrix whose diagonal entries are $\lambda_i$ for $i=1,\ldots,n$. 
\item Let $Q=P^{-1}DP$. This is a positive-semidefinite symmetric matrix which induces the same measured foliation we began with.
\end{enumerate}

Taking the projective class of the matrix gives us the associated element of the Satake compactification. This establishes maps in both directions which are inverses, as required. 
\end{proof} 

The viewpoint of Lemma \ref{measuredflat} gives a geometric way to interpret quadratic forms as measured foliations. Next, we will give the collection $\mathcal{T}(n)\cup\mathcal{PMFF}$ a topology. We do so by giving a notion of convergence to points of $\mathcal{PMFF}$ by sequences of points in $\mathcal{T}(n)=\SL(n,\R)/\SO(n)$. Let $(F,[\mu])\in\mathcal{PMFF},$ where $F$ is the foliation of $\mathcal{T}^n$ and $[\mu]$ is the projective class of the transverse measure. Let $(X_i)_{i\in\N}$ be a sequence of elements of $\mathcal{T}(n)$. 

\begin{defn}
\label{thurstonconvergence}
We say the sequence $(X_i)_{i\in\N}$ \emph{converges} to $(F,\mu)$ if for 
\begin{equation}
\label{convcoeff}
r_i = 1/\max\{\lambda\ :\ \lambda\text{ is an eigenvalue of }X_i\}
\end{equation}
the following holds: there exists a representative $\mu_0\in[\mu]$ such that for all simple closed curves $\gamma\subseteq\mathcal{T}^n$, we have $$\ell_{r_iX_i}(\gamma)\xrightarrow{i\to\infty}\mu_0(\gamma)$$ where $\ell_Q(\gamma)$ denotes the length of the curve $\gamma$ with the metric $Q$. 
\end{defn}

\begin{rem}
Convergence to points of $\mathcal{PMFF}$ may also be viewed more geometrically: we could also define convergence to $\mathcal{PMFF}$ by requiring that the Hausdorff distance between unit balls goes to 0. This is essentially convergence of metrics while allowing some directions to degenerate. 
\end{rem}

\begin{lem}
\label{iscompact}
The collection $\mathcal{T}(n)\cup\mathcal{PMFF}$ is compact.
\end{lem}
\begin{proof}
We show that every sequence has a convergent subsequence. First, suppose $(X_i)_{i\in\N}$ consists only of elements of $\mathcal{T}(n)$, but no subsequence converges to a point of $\mathcal{T}(n)$. Consider then the sequence of matrices $r_iX_i$, where $r_i$ is defined in Equation \ref{convcoeff}. Now, the set of positive-definite symmetric matrices with eigenvalues bounded above by 1 is compact, so we may assume $r_iX_i$ converges to a positive-semidefinite matrix $M$. By Lemma \ref{measuredflat} and by construction, $M$ corresponds to an element of $\mathcal{PMFF}$ which satisfies the conditions of Definition \ref{thurstonconvergence}.

Now suppose that some $X_k\in\mathcal{PMFF}$ for some (perhaps infinitely many) $k\in\N$. Pick a sequence $(Y^k_j)_{j\in\N}\in\mathcal{T}(n)$ which converges to $X_k$. Then replace $X_k$ with $Y^k_k$ in the sequence $(X_i)_{i\in\N}$, and use the first case to find a limit for the new sequence. The original sequence also must converge to this same limit. 
\end{proof}

We are now prepared to make the following definition.

\begin{defn}
The \emph{Thurston compactification} of $\mathcal{T}(n)$ is $$\overline{\mathcal{T}(n)}^{Th}:=\mathcal{T}(n)\cup\mathcal{PMFF}.$$
\end{defn}

By Lemmas \ref{measuredflat} and \ref{iscompact}, we see that the Thurston compactification $\overline{\mathcal{T}(n)}^{Th}$ is a compactification of $\mathcal{T}(n)$ built from measured foliations on the underlying structures, as in Thurston's compactification for hyperbolic surfaces.

\begin{lem}
\label{convergenceequiv}
Let $(F,[\mu])\in\mathcal{PMFF}$, and let $[Q]\in\partial\overline{\mathcal{T}(n)}^{Th}$ be the quadratic form associated to $(F,[\mu])$. For a sequence $(X_i)_{i\in\N}\in\mathcal{T}(n)$, we have $$(X_i)_{i\in\N}\xrightarrow{i\to\infty}(F,[\mu]) \text{ if and only if } (X_i)_{i\in\N}\xrightarrow{i\to\infty} [Q]$$ where on the right-hand side the convergence is with respect to the topology on the Satake compactification. 
\end{lem}
\begin{proof}
Notice that convergence on the right-hand side is equivalent to the following: if $1/r_i$ is the maximal eigenvalue of $X_i$ for each $i$, then $$r_iX_i\xrightarrow{i\to\infty}Q$$ for some representative $Q\in[Q]$ as matrices. Let $\mu_0$ be the representative of $[\mu]$ associated to the semidefinite form $Q$. Then $\ell_Q(\gamma) = \mu_0(\gamma)$ for all simple closed curves $\gamma$, and so from Lemma \ref{measuredflat} we have $$X_i\xrightarrow{i\to\infty}(F,[\mu]).$$ The reverse implication is nearly identical.
\end{proof}

Immediately following from Lemmas \ref{measuredflat} and \ref{convergenceequiv} is the following:

\begin{cor}
\label{thurstonhomeo}
The identity map on $\mathcal{T}(n)$ extends to a homeomorphism $$\overline{\mathcal{T}(n)}^{Th} \cong \overline{\mathcal{T}(n)}^S_{\pi}.$$
\end{cor} 
\begin{proof}
Lemma \ref{convergenceequiv} shows that the bijection from Lemma \ref{measuredflat} preserves convergence in both directions.
\end{proof}

Next, we endow $\mathcal{PMFF}$ with a $\SL(n,\R)$-action. For $g\in\SL(n,\R)$, define:
$$
g\cdot(F,[\mu]) = (Fg, [g^{-1}*\mu]).
$$
One can verify that this defines a $\SL(n,\R)$-action on $\mathcal{PMFF}$. 

\begin{lem}
\label{equivariantthurston}
This $\SL(n,\R)$-action is equivariant with respect to the bijection of Lemma \ref{measuredflat}.
\end{lem}
\begin{proof}
Recall from Lemma \ref{measuredflat} that for any smooth arc $\gamma$, if $(F,\mu_0)$ is a representative of the projective class of $(F,[\mu])$ associated to $Q\in[Q]$, then $$\ell_{Q}(\gamma) = \mu_0(\gamma).$$ Now, for $g\in\SL(n,\R)$ we have $$\ell_{g\cdot Q}(\gamma) = \ell_{Q}(g\cdot\gamma) = \mu_0(g\cdot\gamma) = g^{-1}*\mu_0(\gamma).$$ Finally, if $\gamma$ is a curve contained in a single leaf, then $g\cdot\gamma = \gamma g$ is then contained in a leaf of $g\cdot F = Fg$. 
\end{proof}

Combining Lemma \ref{equivariantthurston} and Corollary \ref{thurstonhomeo}, we arrive at:

\begin{thm}
\label{thurstonGisom}
The Thurston compactification $\overline{\mathcal{T}(n)}^{Th}$ is $\SL(n,\R)$-isomorphic to the Satake compactification with respect to the standard representation $\overline{\mathcal{T}(n)}^S_{\pi}$. 
\end{thm}

\noindent Theorem \ref{horogeomcompact} is then the combined results of Theorem \ref{thurstonGisom} and Proposition \ref{thurstonhoroboundary}.

\bibliography{mybibliography}
\bibliographystyle{alpha}

\end{document}